\documentclass[11pt]{amsart}
\usepackage{graphicx,amssymb,amsmath,amsthm}
\usepackage{enumerate}
\usepackage{comment,cite,color}
\usepackage{cite,color}
\usepackage{algorithmic}
\usepackage{mathrsfs}
\usepackage[ruled]{algorithm}
\usepackage{epsfig}
\usepackage{lscape}

\usepackage{algorithm}  
\usepackage{algorithmic} 

\usepackage{hyperref} 
\hypersetup{
    colorlinks=true, 
    linkcolor=blue,  
    urlcolor=cyan,   
    filecolor=magenta, 
    citecolor=green  
}

\newtheorem{definition}{Definition}[section]
\newtheorem{theorem}{Theorem}[section]
\newtheorem{lemma}[definition]{Lemma}

\newtheorem{problem}{Problem}

\newcommand{\vA}{\mathbf A}

\newcommand{\abs}[1]{\lvert#1\rvert}

\newtheorem{remark}{Remark}[section]

\newcommand{\R}{{\mathbb R}}

\renewcommand{\eqref}[1]{(\ref{#1})}

\newcommand{\vx}{{\mathbf x}}
\newcommand{\vy}{{\mathbf y}}

\newcommand{\va}{{\mathbf a}}

\usepackage{hyperref} 
\hypersetup{
    colorlinks=true, 
    linkcolor=blue,  
    urlcolor=cyan,   
    filecolor=magenta, 
    citecolor=green  
}

\newtheorem{prop}{Proposition}[section]

\newtheorem{coro}[prop]{Corollary}

\date{}

\begin{document}

\title{
New Lower Bounds for the Minimum Singular Value in Matrix Selection
}

\author{Zhiqiang Xu}


\begin{abstract}
The objective of the matrix selection problem is to select a submatrix $\vA_{S}\in \mathbb{R}^{n\times k}$ from $\vA\in \mathbb{R}^{n\times m}$ such that its minimum singular value is maximized. In this paper, we employ the interlacing polynomial method to investigate this problem. This approach allows us to identify a submatrix $\vA_{S_0}\in \mathbb{R}^{n\times k}$ and establish a lower bound for its minimum singular value.
 Specifically, unlike common interlacing polynomial approaches that estimate the smallest root of the expected characteristic polynomial via barrier functions, we leverage the direct relationship between roots and coefficients. 
This leads to a tighter lower bound when $k$ is close to $n$.
For the case where $\vA\vA^{\top}=\mathbb{I}_n$ and $k=n$, our result improves the well-known result by Hong-Pan \cite{HongPan}, which involves extracting a basis from a tight frame and establishing a lower bound for the minimum singular value of the basis matrix.
\end{abstract}

\maketitle

\bigskip \medskip

\section{Introduction}
\subsection{Problem setup}
For a matrix $\vA=(\va_1,\ldots,\va_m)\in \R^{n\times m}$
and a subset $S\subset [m]:=\{1,\ldots,m\}$ of size $k$, we denote by $\vA_S:=(\va_j : j\in S)$
  the submatrix of $\vA$ comprising the columns indexed by the set $S$. 
  Here, $\# S$ denotes the cardinality of $S$. 
  For convenience, we use $\sigma_{\rm min}(\vA_S)$ to denote the $\min\{k,n\}$-th largest singular value of $\vA_S$, i.e.,
\begin{equation}\label{eq:sigma}
\sigma_{\rm min}(\vA_S):= 
\begin{cases}
\min\limits_{\|\vx\|_2=1, \vx\in \R^k} \|\vA_S\vx\| & \text{ if } 1\leq k \leq n, \\
\min\limits_{\|\vx\|_2=1, \vx\in \R^n} \|(\vA_S)^\top\vx\| & \text{ if } n<k\leq m.
\end{cases}
\end{equation}

In the case where $\vA_S$   has full rank, $\sigma_{\rm min}(\vA_S)$ represents the minimum non-zero singular value of $\vA_S$.
In this paper, we focus on the following subset selection problem.
\begin{problem}\label{pr:pr1}
Given a target matrix $\vA\in \R^{n\times m}$ and an integer $k\in [1,m)$, find a subset $S\subset [m]:=\{1,\ldots,m\}$ of size $k$ such that 
$\sigma_{\rm min}(\vA_S)$ is maximized over all possible ${m\choose k}$ 
choices for the subset $S$. Here, $\sigma_{\rm min}(\vA_S)$ is defined in (\ref{eq:sigma}).
\end{problem}

Problem \ref{pr:pr1}  is a pervasive concept that finds relevance across a wide range of fields. Its theoretical underpinnings are evident in areas such as the restricted invertibility principle of Bourgain and Tzafriri \cite{BT87,BT89, Ver01} and Banach space theory \cite{You14}.
Beyond its theoretical significance, the problem has considerable practical implications in various applied domains. These include its use in machine learning \cite{submach},   rank-revealing QR factorizations \cite{HongPan}, sensor selection \cite{subsensor},  and feature selection within $k$-means clustering \cite{subkmeans}.  In the statistics literature, this subset selection problem is also notably linked to experimental design \cite{sub2}.

\subsection{Our contribution}

This paper aims to establish a lower bound for $\max\limits_{S \subset [m], \#S=k} \sigma_{\rm min}(\vA_S)$. Building upon the work of \cite{spielman17} and \cite{xu21}, we utilize interlacing techniques to derive a novel lower bound for $\max\limits_{S \subset [m], \#S=k} \sigma_{\rm min}(\vA_S)$.
Notably, this bound proves sharper than previous results, particularly when $k=\#S$ approaches $n$. This advantage stems from our distinct methodology for estimating the smallest root. Specifically, unlike common approaches in the interlacing polynomial literature that estimate the smallest root of the expected characteristic polynomial via  barrier function  methods \cite{spielman17, xu21, inter2,CXX1,CXX2}, our work instead leverages the direct relationship between roots and coefficients for this estimation. 

Our main result is stated as follows:

\begin{theorem}\label{th:main}
Assume that  $\vA\in \R^{n\times m}$  with ${\rm rank}(\vA)=n$.
Assume that $m\geq n+1$ and $k$ is an integer satisfying $1\leq k\leq m$. 
There exists $S_0\subset [m]$ of size $k$ such that
\begin{equation}\label{eq:thlower1}
 (\sigma_{\rm min}(\vA_{S_0}))^2 \geq \frac{\abs{n-k}+1}{\alpha\cdot \left(\min\{k,n\}\cdot m-n\cdot k\right)+\abs{n-k}+1} (\sigma_{\rm min}(\vA))^2,
\end{equation}
where 
\begin{equation}\label{eq:thlower2}
\alpha:=\alpha_{m,n,k} := \begin{cases}
    \frac{1}{2}+\frac{1}{2} \left(1-h_1(m,n,k)\right)^{1/2} & \text{if } m \geq n + k \text{ and } k\leq n \\
     \frac{1}{2}+\frac{1}{2} \left(1-h_1(m,k,n)\right)^{1/2} & \text{if } m \geq n + k \text{ and } k\geq n \\
   \frac{1}{2}+\frac{1}{2} \left(1- h_2(m,n,k)\right)^{1/2} & \text{if } m \leq n + k \text{ and } k\leq n\\
    \frac{1}{2}+\frac{1}{2} \left(1- h_2(m,k,n)\right)^{1/2} & \text{if } m \leq n + k \text{ and } k\geq n
\end{cases}
\end{equation}
with $h_1$ and $h_2$ given by:
\begin{equation}\label{eq:h12}
\begin{aligned}
h_1(m,x_2,x_3) &= 4 \cdot \frac{x_3-1}{x_3^2} \cdot \left(1 - \frac{(x_3-1) \cdot (m-x_3+1)}{(m-x_2) \cdot x_2}\right), \\
h_2(m,x_2,x_3) &= 4 \cdot \frac{(x_2+1-x_3) \cdot (m-x_2-1) \cdot (x_2+x_3+1-m)}{(m-x_3) \cdot x_3 \cdot (m-x_2)^2}.
\end{aligned}
\end{equation}
\end{theorem}

As said before, the proof of Theorem \ref{th:main} employs the interlacing polynomials method. Our primary theoretical contribution lies in refining the estimation of the smallest root of an associated expected polynomial. This specific refinement of the root estimation does not alter the fundamental  design of the algorithm based on the proof of Theorem \ref{th:main}. Consequently, the algorithm can be constructed following the standard procedure for algorithms based on interlacing polynomials \cite{spielman17, xu21}. While we omit a detailed exposition of its well-established details, we can now state its performance characteristics.
The algorithm achieves a polynomial time complexity of $O(k(m-\frac{k}{2})n^{\theta}\log(1/\epsilon))$, where $\theta \in (2, 2.373)$ denotes the matrix multiplication complexity exponent \cite{spielman17, xu21}. For any $\epsilon \in (0,\frac{1}{k})$, it outputs an $S_0\subset [m]$ of size $k$ satisfying the following bound:
\[
 (\sigma_{\rm min}(\vA_{S_0}))^2 \geq \frac{\abs{n-k}+1}{\alpha\cdot \left(\min\{k,n\}\cdot m-n\cdot k\right)+\abs{n-k}+1}(1-k\epsilon) (\sigma_{\rm min}(\vA))^2.
\]
The factor $1-k\epsilon$ in the inequality arises due to the computational approximation inherent in the algorithm.

\begin{remark}
It should be noted that the full-rank requirement for $\vA$ in Theorem \ref{th:main} is not a strict necessity. If $\text{rank}(\vA) = r$, the theorem is still applicable by simply substituting $n$ with $r$ in the expression for the lower bound.
\end{remark}

\begin{remark}
The $\alpha$ defined in (\ref{eq:thlower2}) satisfies $0<\alpha<1$, except for the case where $\alpha=1$ if $k=1$ or $m=n+1$ or $m=k+1$.
\end{remark}

Applying Theorem \ref{th:main} with $k=n$, we derive the following corollary, which establishes the existence of a basis in $\vA$, denoted as $\vA_{S_0}$, and provides a lower bound for its minimum singular value, $\sigma_{\min}(\vA_{S_0})$.

\begin{coro}\label{th:main1}
Assume that $\vA\in \R^{n\times m}$ satisfying $\vA\vA^\top={\mathbb I}_n$, where ${\mathbb I}_n$ denotes the $n \times n$ identity matrix.
Assume that $m\geq n+1$.
There exists $S_0\subset [m]$ of size $n$ such that
\[
 \sigma_{\rm min}(\vA_{S_0})^2 \geq \frac{1}{\alpha\cdot n\cdot  \left( m-n\right)+1},
\]
where 
\begin{equation}\label{eq:alpha1}
\alpha:=\alpha_{m,n} := \begin{cases}
    \frac{1}{2}+\frac{1}{2} \left(1-h_1(m,n,n)\right)^{1/2} & \text{if } m \geq 2n   \\
   \frac{1}{2}+\frac{1}{2} \left(1- h_2(m,n,n)\right)^{1/2} & \text{if } m \leq 2n  \\
\end{cases},
\end{equation}
with $h_1$ and $h_2$ given by:
\begin{equation}\label{eq:h12}
\begin{aligned}
h_1(m,n,n) &= 4 \cdot \frac{n-1}{n^2}\left(1-\frac{(n-1)(m-n+1)}{(m-n)n}\right)\\
h_2(m,n,n) &= 4 \cdot \frac{(m-n-1)(2n+1-m)}{(m-n)^3n}.
\end{aligned}
\end{equation}
\end{coro}

\begin{remark}
The $\alpha$ defined in (\ref{eq:alpha1}) satisfies $0<\alpha<1$ provided $m\geq n+2$ and $n\geq 2$.
To provide further insight into the behavior of $\alpha$, we highlight the asymptotic properties of $h_1$ and $h_2$.
For large $m$, $\alpha$ is given by $ \frac{1}{2}+\frac{1}{2} \left(1-h_1(m,n,n)\right)^{1/2}$, and its asymptotic expansion is:
\[
\alpha =  \frac{1}{2}+\frac{1}{2} \left(1-h_1(m,n,n)\right)^{1/2} = 1 - \frac{1}{n^2} + O\left(\frac{1}{n^3}\right).
\]
Furthermore, in the range $n+1 \leq m \leq 2n$, where $\alpha$ is given by $ \frac{1}{2}+\frac{1}{2} \left(1-h_1(m,n,n)\right)^{1/2}$, $\alpha$ approaches 1 as $m-n$ increases, with its rate characterized by:
\[
\alpha =  \frac{1}{2}+\frac{1}{2} \left(1-h_1(m,n,n)\right)^{1/2} = 1 - \frac{1}{(m-n)^2} + O\left(\frac{1}{(m-n)^3}\right).
\]
\end{remark}

The subsequent theorem offers a refined lower bound, particularly for cases where $k \in \{2,3\}$ or $n \in \{2,3\}$. While our results for $k=2$ or $n=2$ align with Theorem \ref{th:main}, we present a simplified proof. Furthermore, for $k=3$ or $n=3$, our bounds yield a significant improvement over Theorem \ref{th:main}.

\begin{theorem}\label{th:kn23}
Assume that $\vA\in \R^{n\times m}$ with ${\rm rank}(\vA)=n$.
There exists $S_0\subset [m]$ of size $k$ such that
\[
 (\sigma_{\rm min}(\vA_{S_0}))^2 \geq \begin{cases}
      g_1(n)\cdot (\sigma_{\rm min}(\vA))^2 & \text{if }  k=2, n\geq k, \\
  g_2(n)\cdot (\sigma_{\rm min}(\vA))^2 & \text{if }  k=3, n\geq k,\\
   g_1(k)\cdot (\sigma_{\rm min}(\vA))^2 & \text{if }  n=2,  k\geq n,\\
   g_2(k)\cdot (\sigma_{\rm min}(\vA))^2 & \text{if }  n=3, k\geq n,\\
\end{cases}
\]
where $g_1(x)=\frac{x-\sqrt{x(m-x)/(m-1)}}{m}$ and 
$g_2(x)=\frac{x}{m}+\frac{2}{m}\sqrt{\frac{x(m-x)}{m-1}}\cos \left(\frac{1}{3} \arccos\left(\frac{m-2x}{m-2}\sqrt{\frac{m-1}{x(m-x)}}\right) +\frac{2\pi}{3}\right)$.
\end{theorem}

\subsection{Related work}

This subsection presents related work and compares our results with those in the literature. A common characteristic of prior results is the imposition of specific constraints on the range of $k$, such as $k=n$, $k\in [1, n]$, or $k\in [n, m-1]$. In contrast, our result in Theorem \ref{th:main} is more general, imposing no constraints on the integer $k$ beyond the natural condition $k\in [1, m-1]$. Consequently, we conduct comparisons against previous results across various ranges of $k$. Our analysis reveals that, particularly when $k$ is close to $n$, our bounds demonstrate superiority over existing ones.

{\bf The case where $k=n$.}
Assume that $\vA\in \R^{n\times m}$ satisfying $\vA\vA^\top={\mathbb I}_n$
In \cite{HongPan}, a well-known result was established, demonstrating the existence of $S_0\subset [m]$ of size $n$ such that 
\begin{equation}\label{eq:hongpan1}
(\sigma_{\rm min}(\vA_{S_0}))^2\geq \frac{1}{n(m-n)+1}.
\end{equation}
 This finding was subsequently employed to provide a constructive proof for the existence of rank-revealing QR factorizations. Furthermore, for the specific case where $n=2$, the authors of \cite{HongPan} also showed that there exists $S_0\subset [m]$ of size $2$ satisfying 
 \begin{equation}\label{eq:hongpan2}
 (\sigma_{\rm min}(\vA_{S_0}))^2\geq {(2-\sqrt{2}\sqrt{(m-2)/(m-1)})/m}.
 \end{equation}
Corollary \ref{th:main1} of this paper improves the bound to $(\sigma_{\rm min}(\vA_{S_0}))^2\geq \frac{1}{\alpha n(m-n)+1}$, which is a stronger result since $0<\alpha<1$ provided that $n\geq 2$ and $m\geq n+2$. The lower bound in (\ref{eq:hongpan2}) aligns with Theorem \ref{th:kn23} when $n=k=2$. However, their proof spans two pages, whereas our proof is significantly simpler.

{\bf The case where $n\leq k\leq m-1$.}
Assume that  $\vA\in \R^{n\times m}$  with ${\rm rank}(\vA)=n$.
Assume that $m\geq n+1$ and $k$ is an integer satisfying $n\leq k\leq m-1$. 
For $n\leq k\leq m-1$, the lower bound (\ref{eq:thlower1}) reduces to: 
\begin{equation}\label{eq:thlowerkn}
 (\sigma_{\rm min}(\vA_{S_0}))^2 \geq \frac{k-n+1}{\alpha\cdot n\left( m- k\right)+k-n+1} (\sigma_{\rm min}(\vA))^2.
\end{equation}
We next compare the lower bound in (\ref{eq:thlowerkn}) with the respective lower bounds presented in \cite{sub1} and \cite{xu21}.

In \cite{sub1}, a greedy algorithm is introduced that identifies a subset $S_0 \subset [m]$ of size $k$ satisfying:
\begin{equation}\label{eq:greedy}
 (\sigma_{\rm min}(\vA_{S_0}))^2 \geq \frac{k-n+1}{ n\left( m- k\right)+k-n+1} (\sigma_{\rm min}(\vA))^2.
\end{equation}
The inequality $0<\alpha<1$ is valid when $k\leq m-2$ and $n\geq 2$. Since these conditions are typically met in practical scenarios, the lower bound in (\ref{eq:thlowerkn}) is tighter than that in (\ref{eq:greedy}).

It was shown in \cite{xu21} that there exists  $S_0\subset [m]$ of size $k$ such that
\begin{equation}\label{eq:lower1}
 (\sigma_{\rm min}(\vA_{S_0}))^2 \geq \frac{(\sqrt{(k+1)(m-n)}-\sqrt{n(m-k-1)})^2}{m^2} (\sigma_{\rm min}(\vA))^2,
\end{equation}
Noting that 
\[
\begin{aligned}
\frac{(\sqrt{(k+1)(m-n)}-\sqrt{n(m-k-1)})^2}{m^2}&=\frac{(k-n+1)^2}{(\sqrt{(k+1)(m-n)}+\sqrt{n(m-k-1)})^2}\\
&\leq  \frac{(k-n+1)^2}{(k+1)(m-4n)+3mn}
\end{aligned}
\]
we can show that 
\[
\frac{k-n+1}{\alpha\cdot n\left( m- k\right)+k-n+1}> \frac{(\sqrt{(k+1)(m-n)}-\sqrt{n(m-k-1)})^2}{m^2}
\]
provided $n\leq k\leq n+3<m$. This implies that the lower bound presented in (\ref{eq:thlowerkn}) is tighter than that in (\ref{eq:lower1}) for $k \in [n,n+3]$.

\subsection{The case where $k\leq n$.}

Assume that  $\vA\in \R^{n\times m}$  satisfies $\vA\vA^\top={\mathbb I}_n$.
Assume that $m\geq n+1$ and $k$ is an integer satisfying $k<n$. 
It was shown in \cite{spielman17} that there exists  $S_0\subset [m]$ of size $k$ such that
\begin{equation}\label{eq:lower21}
 (\sigma_{\rm min}(\vA_{S_0}))^2 \geq \left(1-\sqrt{\frac{k}{n}}\right)^2\frac{n}{m}.
\end{equation}
For $k\leq n$, the lower bound (\ref{eq:thlower1}) reduces to: 
\begin{equation}\label{eq:thlowerkn21}
 (\sigma_{\rm min}(\vA_{S_0}))^2 \geq \frac{n-k+1}{\alpha\cdot k\cdot (m-n)+n-k+1}.
\end{equation}
A direct calculation shows that the following holds provided $k\in [n-4, n]$:
\[
 \frac{n-k+1}{\alpha\cdot k\cdot (m-n)+n-k+1}>\left(1-\sqrt{\frac{k}{n}}\right)^2\frac{n}{m}.
\]
It follows that the lower bound established in (\ref{eq:thlowerkn21}) provides a tighter estimate than (\ref{eq:lower21}) when $k \in [n-4,n]$.

\section{Interlacing families }

Here we recall the definition and related results of interlacing families from \cite{inter1, inter2, spielman17}. Throughout this paper, we say that a univariate polynomial is \emph{real-rooted} if all of its coefficients and roots are real.
\begin{definition}{\rm (see \cite[Definition 3.1]{inter2} ) }
We say a real-rooted polynomial $g(x) =\alpha_0\prod_{i=1}^{n-1}(x-\alpha_i)$ \emph{interlaces} a real-rooted polynomial $f(x) =\beta_0\prod_{i=1}^{n}(x - \beta_i)$ if $\beta_1\leq\alpha_1\leq\beta_2\leq\alpha_2\leq\cdots\leq \alpha_{n-1}\leq \beta_n $.
For polynomials $f_1,\ldots,f_k$, if there exists a polynomial $g$ that interlaces $f_i$ for each $i$, then we say that $f_1,\ldots,f_k$ have a \emph{common interlacing}.
\end{definition}

Following \cite{spielman17},  we define the concept of an interlacing family of polynomials as follows:
\begin{definition}[\cite{spielman17}, Definition 2.5]
\label{def-inter} An {\em interlacing family} consists of a finite rooted tree
$\mathbb{T}$ and a labeling of the nodes $v\in \mathbb{T}$ by monic real-rooted
polynomials $f_{v}(x)\in\mathbb{R}[x]$,  satisfying two key properties:
\begin{itemize}
  \item[(a)] Every polynomial $f_{v}(x)$ corresponding to a non-leaf node $v$ is a convex combination of the polynomials corresponding to the children of $v$.
  \item[(b)] For all nodes $v_1,v_2\in \mathbb{T}$ with a common parent, the polynomials $f_{v_1}(x),f_{v_2}(x)$ have a common interlacing.
\end{itemize}
We say that a set of polynomials form an interlacing family if they are the labels of
the leaves of ${\mathbb T}$.
\end{definition}

Definition \ref{def-inter} stipulates that polynomials constituting an interlacing family must be of equal degree.
We state the main property of interlacing families in the following lemma.
\begin{lemma}\label{interlacing2}{\rm (see \cite[Theorem 2.7]{spielman17} ) }
Let $F$ be an interlacing family of degree $n$  polynomials with root labeled by $f_\emptyset$
and leaves by $\{f_{\ell} \}_{\ell\in L}$.  For all indices $j\in [1,n]$  there exist leaves $\ell_0, \ell_1\in L$ such that
\[
r_{j}(f_{\ell_0}) \geq r_{j}(f_{\emptyset}) \geq r_{j}(f_{\ell_1}) ,
\]
where  $f_\emptyset$ denotes the polynomial corresponding to 
 the root node of the tree $\mathbb{T}$ and  $r_{j}(f)$ denotes the  $j$th largest   root of the real-rooted polynomial $f$.
\end{lemma}

We assume that $Y\in \R^{n\times m}$ satisfying $YY^\top =\mathbb{I}_{n}$. Here, $\mathbb{I}_n$ denotes the $n\times n$ identity matrix. We use $\vy_1,\ldots,\vy_m$ to denotes the columns of $Y$. Hence, 
\[
\sum_{j=1}^m\vy_j\vy_j^\top\,\,=\,\,\mathbb{I}_n.
\]
For any $S\subset [m]:= \{1,\ldots,m\}$ with $\# S=k$, 
we define the submatrix $Y_S\in \R^{n\times k}$ as:
\[
Y_S\,\,:=\,\, [\vy_j : j\in S]\in \R^{n\times k}.
\] 
The following lemmas are helpful for our argument.
\begin{lemma}\label{le:leafs} (\cite[Theorem 5.4]{spielman17})
We assume that $Y\in \R^{n\times m}$ satisfying $YY^\top =\mathbb{I}_{n}$.
The set of polynomials defined by
\begin{equation}\label{eq:pk}
 \{ p_S(x)=\det(x\mathbb{I}_n-Y_SY_S^\top) : S\subset [m], \#S=k\} 
 \end{equation}
 forms an interlacing family.
\end{lemma}

\begin{lemma}\label{le:fempty} (\cite[Lemma 5.3]{spielman17})
We assume that $Y\in \R^{n\times m}$ satisfying $YY^\top =\mathbb{I}_{n}$.
For $k\in [m]$, set
\begin{equation}\label{eq:fkong}
f_\emptyset(x):=\frac{1}{{m\choose k}} \sum_{S\subset [m], \#{S}=k}p_S(x),
\end{equation}
where
$
p_S(x)=\det(x\mathbb{I}_n-Y_SY_S^\top)
$.
The polynomial $f_\emptyset$, corresponding to the root node of the interlacing family (\ref{eq:pk}), is explicitly given by:
\begin{equation}\label{eq:fempty}
f_\emptyset(x)=\frac{(m-k)!}{m!}(x-1)^{-(m-n-k)}\partial_x^k(x-1)^{m-n}x^n.
\end{equation}
\end{lemma}

\section{Proof of Theorem \ref{th:main}}

A key step in the proof of Theorem \ref{th:main} is to estimate the smallest non-zero root of $f_\emptyset$.
For $1\leq k\leq n$, to eliminate the effect of zero roots of $f_\emptyset$, we typically  estimate the smallest root of $g_\emptyset$, which is defined by
\begin{equation}\label{eq:gkong}
g_\emptyset(x):=(-1)^k\cdot \frac{m!}{(m-k)!}\cdot  f_\emptyset(x)/x^{n-k},
\end{equation}
where $f_\emptyset$ is defined in (\ref{eq:fempty}).
To present the proof of Theorem \ref{th:main}, we first introduce the following lemmas, postponing their proofs until later. 
For convenience, we define $\lambda_{\rm min}(p)$ as the smallest root of a real-rooted polynomial $p$.
\begin{lemma}\label{le:gmin}
Assume that $k, n,$ and $m$ are integers such that $1\leq k\leq n$ and $m\geq n+1$.
The 
$
g_\emptyset(x)
$
is a polynomial with degree of $k$, where $g_\emptyset$ is defined in (\ref{eq:gkong}).
Then the smallest root of $g_\emptyset$ satisfies
\[
 \lambda_{\rm min}(g_\emptyset) \,\,\geq \,\,\frac{n-k+1}{\alpha\cdot k(m-n)+n-k+1},
\]
where  
\[
\alpha:=\alpha_{m,n,k} := \begin{cases}
    \frac{1}{2}+\frac{1}{2} \left(1-h_1(m,n,k)\right)^{1/2} & \text{if } m \geq n + k,  \\
   \frac{1}{2}+\frac{1}{2} \left(1- h_2(m,n,k)\right)^{1/2} & \text{if } m < n + k, \\
\end{cases}
\]
with $h_1, h_2$ being defined in (\ref{eq:h12}).
\end{lemma}

\begin{lemma}\label{le:knh1}
Assume that $k, n,$ and $m$ are integers such that $n\leq k\leq m$ and $m\geq n+1$.
 Let $f_\emptyset$ be defined in (\ref{eq:fempty}).
 Then the smallest root of $f_\emptyset$ satisfies
 
 \[
 \lambda_{\rm min}({f}_\emptyset) \,\,\geq \,\,\frac{k-n+1}{\alpha\cdot n(m-k)+k-n+1},
\]
where 
\[
\alpha:=\alpha_{m,n,k} := \begin{cases}
    \frac{1}{2}+\frac{1}{2} \left(1-h_1(m,k,n)\right)^{1/2} & \text{if } m \geq n + k,  \\
   \frac{1}{2}+\frac{1}{2} \left(1- h_2(m,k,n)\right)^{1/2} & \text{if } m < n + k, \\
\end{cases}
\]
 with $h_1, h_2$ being defined in (\ref{eq:h12}).
\end{lemma}

We next present the proof of Theorem  \ref{th:main}.
\begin{proof}[Proof of Theorem \ref{th:main}]

We first establish the conclusion by assuming that $\vA\vA^\top={\mathbb I}_n$.
For this scenario, we first focus on the condition $1\leq k\leq n$.
For $S\subset [m]$ with $\# S=k$, set 
\[
 p_S(x)\,\,:=\,\, \det(x\mathbb{I}_n-\vA_S\vA_S^\top).
\]
Since $k\leq n$, we have
$$
(\sigma_{\rm mim}(\vA_S))^2=r_k(p_S),
$$
where $r_k(p_S)$ denotes the $k$th largest root of $p_S$.
According to Lemma \ref{le:leafs}, the set of polynomials defined by
\[
 \{ p_S(x)=\det(x\mathbb{I}_n-\vA_S\vA_S^\top) : S\subset [m], \#S=k\} 
 \]
 forms an interlacing family.

Let $f_\emptyset$ denote the polynomial corresponding to the root node of the interlacing family. 
Lemma \ref{le:fempty} implies 
\begin{equation*}
f_\emptyset(x)=\frac{(m-k)!}{m!}(x-1)^{-(m-n-k)}\partial_x^k(x-1)^{m-n}x^n.
\end{equation*}
A simple observation is that, since $f_\emptyset$ contains the factor $x^{n-k}$, it follows that $r_j(f_\emptyset)=0$ for $k+1\leq j\leq n$.
Then, according to Lemma   \ref{le:gmin},
\begin{equation}\label{eq:ourlow}
r_k(f_\emptyset) =\lambda_{\rm min}(g_\emptyset) \geq \frac{n-k+1}{\alpha\cdot k(m-n)+n-k+1},
\end{equation}
where $g_\emptyset$ and $\alpha$ are defined in Lemma \ref{le:gmin}.  
Combining Lemma \ref{interlacing2} with \eqref{eq:ourlow}, it follows that there exists an 
 there exists $S_0\subset [m]$ with $\#S_0=k$ such that
\begin{equation}\label{eq:ks01}
(\sigma_{\rm min}(\vA_{S_0}))^2=r_k(p_{S_0}) \geq r_k(f_\emptyset)  \geq \frac{n-k+1}{\alpha\cdot k(m-n)+n-k+1}.
\end{equation}

We now consider the case where $k \geq n$. 
For any $S\subset [m]$ with $\# S=k$, set 
\[
 p_S(x)\,\,:=\,\, \det(x\mathbb{I}_n-\vA_S\vA_S^\top).
\]
Given $k\geq n$, it holds that
$$
(\sigma_{\rm mim}(\vA_S))^2=\lambda_{\rm min}(p_S).
$$

Employing a similar argument as before and applying Lemma \ref{le:knh1}, we arrive at the following conclusion:
there exists $S_0\subset [m]$ with $\#S_0=k$ such that
\begin{equation}\label{eq:ks02}
(\sigma_{\rm mim}(\vA_S))^2= \lambda_{\rm min}(p_{S_0}) \geq \lambda_{\rm min}(f_\emptyset)  \geq  \frac{k-n+1}{\alpha\cdot n(m-k)+k-n+1},
\end{equation}
where  $\alpha$ is defined in Lemma \ref{le:knh1}.
Synthesizing the results from equations (\ref{eq:ks01}) and (\ref{eq:ks02}), we can establish the theorem for the case where $\vA\vA^\top={\mathbb I}_n$.

We now consider the general matrix  $\vA\in \R^{n\times m}$.
We assume that the thin SVD of $\vA$ is given by $\vA=U\Sigma Y$, where 
$\Sigma={\rm diag}(\sigma_1(\vA),\ldots,\sigma_n(\vA))\in \R^{n\times n}$
is a diagonal matrix containing the singular values of $\vA$, $U\in \R^{n\times n}$  and $Y\in \R^{n\times m}$
  satisfy $U^\top U=YY^\top={\mathbb I}_n$.
  A straightforward argument reveals the following relationship for any subset $S\subset [m]$:
\begin{equation}\label{eq:noninstro} 
\sigma_{\rm min}(\vA_{S})^2 \geq (\sigma_{\rm min}(Y_{S}))^2 (\sigma_{\rm min }(\vA))^2. 
\end{equation}
To demonstrate this, let $k := \# S$ denote the cardinality of the subset $S$.
For the case where $k \geq n$  we have
\[
\begin{aligned}
\sigma_{\rm min}(\vA_{S})^2 &=\min_{\vx\in \R^n, \|\vx\|=1}\|\vA_S \vx\|^2 = \min_{\vx\in \R^n, \|\vx\|=1} \|U\Sigma Y_S\vx\|^2 =
\min_{\vx\in \R^n, \|\vx\|=1} \|\Sigma Y_S\vx\|^2\\
&\geq  (\sigma_{\rm min }(\vA))^2  \min_{\vx\in \R^n, \|\vx\|=1} \| Y_S\vx\|^2
=(\sigma_{\rm min}(Y_{S}))^2 (\sigma_{\rm min }(\vA))^2. 
\end{aligned}
\]
The argument proceeds similarly for the case where $k < n$.
 Since $Y$ is an isotropic matrix ($Y \in \R^{n \times m}$ with $YY^\top={\mathbb I}_n$), we can apply equations (\ref{eq:ks01}) and (\ref{eq:ks02}) to $(\sigma_{\rm min}(Y_S))^2$. This result, combined with the derived inequality (\ref{eq:noninstro}), allows us to establish a lower bound for $\sigma_{\rm min}(\vA_S)^2$, which concludes the proof.
\end{proof}

\begin{proof}[Proof of Lemma \ref{le:gmin}]
According to equation (\ref{eq:fempty}), we can express 
\[
g_\emptyset(x):=(-1)^k\cdot \frac{m!}{(m-k)!}\cdot f_{\emptyset}(x)/x^{n-k}
\]
 in the form:
\begin{equation}\label{eq:gkong}
g_\emptyset(x)=c_0 B_{k,0}(x)+c_1 B_{k,1}(x)+\cdots+c_kB_{k,k}(x)
\end{equation}
where the coefficients $c_i$ are given by:
\[
c_i=
\begin{cases} 
(-1)^{i}\frac{n!}{(n-k+i)!}\frac{(m-n)!}{(m-n-i)!} & \text{for } 0\leq i \leq \min \{m-n,k\}, \\ 
0 & \text{for } m-n<i\leq k. 
\end{cases}
\]
Here,   $B_{k,i}$, $i=0,\ldots,k,$ denote the Bernstein basis polynomials, which are defined by:
\[
B_{k,i}(x)={k\choose i} x^i (1-x)^{k-i},\quad i=0,\ldots,k.
\]
A straightforward calculation reveals that the coefficient of $x^k$
  in $g_\emptyset(x)$
  is:
\[
(-1)^k\sum_{i=0}^{\min\{m-n,k\}} {k\choose i} \frac{n!}{(n-k+i)!}\cdot \frac{(m-n)!}{(m-n-i)!}\neq 0.
\]
This non-zero coefficient confirms that $g_\emptyset$ is indeed a polynomial of degree exactly $k$.

Let us first examine the case where $m \geq n+k$. In this scenario, the constant term and the coefficient of $x$ in $g_\emptyset$
  are $c_0$  and $-k(c_0-c_1)$, respectively. Thus, we can express $g_\emptyset(x)$
  as:
\[
g_\emptyset(x)=c_0-k(c_0-c_1)x+O(x^2).
\]

It follows from the mean value theorem and (\ref{eq:fempty}) that all roots of $g_\emptyset$ are in $(0, 1)$.
We use $0<\lambda_1\leq \lambda_2\leq \cdots\leq \lambda_k<1$ to denote the roots of $g_\emptyset$. Then, according to 
Vieta's formulas, we have
\begin{equation}\label{eq:dao1}
\frac{1}{\lambda_1}+\frac{1}{\lambda_2}+\cdots+\frac{1}{\lambda_k}=k\cdot \frac{c_0-c_1}{c_0}=k\cdot \frac{m-k+1}{n-k+1}.
\end{equation}
We observe that $g_\emptyset(1 - x)$ constitutes a polynomial in $x$ of degree $k$. Consequently, $1-\lambda_1, 1-\lambda_2, \ldots, 1-\lambda_k$ are the roots of $g_\emptyset(1-x)$, where $1-\lambda_1 \geq 1-\lambda_2 \geq \cdots \geq 1-\lambda_k$. According to (\ref{eq:gkong}), we have
\[
g_\emptyset(1-x) = c_k - k(c_k - c_{k-1})x + O(x^2).
\]
Applying Vieta's formulas, we obtain
\begin{equation}\label{eq:dao2}
\frac{1}{1-\lambda_1} + \frac{1}{1-\lambda_2} + \cdots + \frac{1}{1-\lambda_k} =k\cdot \frac{c_k-c_{k-1}}{c_k}= k\cdot \frac{m-k+1}{m-n-k+1}.
\end{equation}
Utilizing equations (\ref{eq:dao1}) and (\ref{eq:dao2}), we derive:
\begin{equation*}
\begin{aligned}
\frac{1}{\lambda_1}+\frac{k-1}{\lambda_k}&\,\,\leq \,\,k\cdot \frac{m-k+1}{n-k+1}\\
 \frac{k-1}{1-\lambda_1}+\frac{1}{1-\lambda_k} &\,\,\leq\,\, k\cdot \frac{m-k+1}{m-n-k+1}.
\end{aligned}
\end{equation*}
These inequalities lead to: 
\begin{equation}
\begin{aligned}
\frac{1}{\lambda_k}-1 &\,\,\leq\,\, \frac{k}{k-1}\cdot \frac{m-k+1}{n-k+1}-\frac{1}{k-1}\cdot \frac{1}{\lambda_1}-1\\
\frac{1}{\lambda_k}-1 &\,\,\geq\,\,  \frac{1}{k\cdot \frac{m-k+1}{m-n-k+1}-1-\frac{k-1}{1-\lambda_1}}.
\end{aligned}
\end{equation}
Consequently, we obtain
\begin{equation}\label{eq:qerci}
\frac{1}{k\cdot \frac{m-k+1}{m-n-k+1}-1-\frac{k-1}{1-\lambda_1}}\leq \frac{k}{k-1}\cdot \frac{m-k+1}{n-k+1}-\frac{1}{k-1}\cdot \frac{1}{\lambda_1}-1.
\end{equation}
Equation (\ref{eq:qerci}) represents a quadratic inequality in terms of $1/\lambda_1$. By solving this inequality, we can determine that
\[
\begin{aligned}
\lambda_{\rm min}(g_\emptyset)=\lambda_1&\geq \frac{2kn(n-k+1)}{kn(k(m-n)+2(n-k+1))+2k(m-n)n\sqrt{(k/2-1)^2+(k-1)^2\frac{(m-k+1)}{(m-n)n}}}\\
&= \frac{n-k+1}{\alpha\cdot k(m-n)+n-k+1},
\end{aligned}
\]
where 
 $\alpha=\frac{1}{2}+\frac{1}{2} \left(1-h_1(m,n,k)\right)^{1/2} $.
 Thus, we have established the lemma's conclusion for the case where $m \geq n+k$.

We now turn our attention to the case where $m<n+ k$.
Let us define 
\[
\tilde{g}_\emptyset(x) := (x-1)^{m-n-k}g_\emptyset(x),
\]
 which is a polynomial of degree $m-n$. 
A key observation is that $\lambda_{\min}(g_{\emptyset}) = \lambda_{\min}(\tilde{g}_{\emptyset})$. Therefore, it suffices to analyze $\lambda_{\min}(\tilde{g}_{\emptyset})$. 
Let $0<\lambda_1 \leq \lambda_2 \leq \cdots \leq \lambda_{m-n}<1$ denote the roots of $\tilde{g}_{\emptyset}$. 
Employing an approach analogous to our previous analysis, we derive the following equations:
\begin{equation}
\begin{aligned}
\frac{1}{\lambda_1} +\frac{1}{\lambda_2} \cdots + \frac{1}{\lambda_{m-n}} &= (m-n)\cdot \frac{n+1}{n-k+1},\\
\frac{1}{1-\lambda_1}+\frac{1}{1-\lambda_2}+\cdots+\frac{1}{1-\lambda_{m-n}}&=(m-n)\cdot \frac{n+1}{k-m+n+1},
\end{aligned}
\end{equation}
which subsequently yield the inequalities:
\[
\begin{aligned}
\frac{1}{\lambda_1}+\frac{m-n-1}{\lambda_{m-n}}&\,\,\leq\,\, (m-n)\cdot\frac{n+1}{n-k+1}\\
\frac{m-n-1}{1-\lambda_1}+\frac{1}{1-\lambda_{m-n}}&\,\,\leq\,\, (m-n)\cdot\frac{n+1}{k-m+n+1}.
\end{aligned}
\]
From these inequalities, we deduce:
\[
\begin{aligned}
\frac{1}{\lambda_{m-n}}-1&\leq \frac{m-n}{m-n-1} \frac{n+1}{n-k+1}-\frac{1}{m-n-1}\frac{1}{\lambda_1}-1,\\
\frac{1}{\lambda_{m-n}}-1& \geq \frac{1}{(m-n)\frac{n+1}{k-m+n+1}-1-\frac{m-n-1}{1-\lambda_1}},
\end{aligned}
\]
which consequently implies: 
\begin{equation}\label{eq:qerci2}
\frac{1}{(m-n)\frac{n+1}{k-m+n+1}-1-\frac{m-n-1}{1-\lambda_1}}\leq \frac{m-n}{m-n-1} \frac{n+1}{n-k+1}-\frac{1}{m-n-1}\frac{1}{\lambda_1}-1.
\end{equation}
Equation (\ref{eq:qerci2}) represents a quadratic inequality in terms of $1/\lambda_1$. Upon solving this inequality, we obtain:
\[
\begin{aligned}
\lambda_{\rm min}(g_{\emptyset})=\lambda_1&\geq \frac{n-k+1}{\frac{1}{2}k(m-n)+n-k+1+\frac{1}{2} k(m-n)\sqrt{ 1- \frac{4(n+1-k)(m-n-1)(n+k+1-m) }{(m-k)k(m-n)^2}  }}\\
&=\frac{n-k+1}{\alpha\cdot k(m-n)+n-k+1},
\end{aligned}
\]
where 
 $\alpha=\frac{1}{2}+\frac{1}{2} \left(1-h_2(m,n,k)\right)^{1/2} $.
\end{proof}

We present the following lemma, originally established in \cite{spielman17}, which proves instrumental in demonstrating Lemma  \ref{le:knh1}.
\begin{lemma}\label{le:knh}(\cite{spielman17})
For any integer $k$ such that $n \leq k \leq m$, 
we have
\[
\partial_x^k (x-1)^{m-n}x^n\,\,=\,\,\frac{(m-n)!}{(m-k)!} \frac{1}{x^{k-n}}\partial_x^n (x-1)^{m-k}x^k.
\]
\end{lemma}

\begin{proof}[Proof of Lemma \ref{le:knh1} ]
According to (\ref{eq:fempty}) and Lemma \ref{le:knh}, we have
 \[
 \lambda_{\rm min}(f_\emptyset)= \lambda_{\rm min}(\tilde{g}_\emptyset),
\]
where $\tilde{g}_\emptyset(x)=\frac{1}{x^{k-n}}\partial_x^n (x-1)^{m-k}x^k$.
A lower bound for 
\[
\lambda_{\rm min}(g_\emptyset) = \lambda_{\rm min}\left(\frac{1}{x^{n-k}}\partial_x^k (x-1)^{m-n}x^n\right)
\]
 is provided by Lemma \ref{le:gmin}.
We apply Lemma \ref{le:gmin} with the symbols $k$ and $n$ interchanged, yielding:
\[
 \lambda_{\rm min}(f_\emptyset)=  \lambda_{\rm min}(\tilde{g}_\emptyset) \,\,\geq \,\,\frac{k-n+1}{\alpha\cdot n(m-k)+k-n+1},
\]
where 
\[
\alpha:=\alpha_{m,n,k} := \begin{cases}
    \frac{1}{2}+\frac{1}{2} \left(1-h_1(m,k,n)\right)^{1/2} & \text{if } m \geq n + k,  \\
   \frac{1}{2}+\frac{1}{2} \left(1- h_2(m,k,n)\right)^{1/2} & \text{if } m < n + k. \\
\end{cases}
\]
This modification is essential because Lemma \ref{le:gmin} originally assumes $k \leq n$, whereas our current scenario requires $n \leq k$.
\end{proof}

\section{Proof of Theorem \ref{th:kn23}}

\begin{proof}
The proof follows a similar structure to that of Theorem \ref{th:main}, however, in this case, we can explicitly present $r_k(f_\emptyset)$
 (or $\lambda_{\rm min}(f_\emptyset)$). Here,  $f_\emptyset$ is defined in Lemma \ref{le:fempty} and $r_k(f_\emptyset)$ denotes the $k$-th largest root of $f_\emptyset$.
   Employing the same reasoning as in the proof of Theorem \ref{th:main}, it suffices to consider the case where $\vA\vA^\top={\mathbb I}_n$.

We first consider the case where $n=2$ and $k\geq n$.
According to the proof of Theorem \ref{th:main}, it is enough to consider $\lambda_{\rm min}(f_\emptyset)$.
For $n=2$,  Lemma \ref{le:fempty} implies 
\[
\begin{aligned}
f_\emptyset(x)&=\frac{(m-k)!}{m!} \frac{(m-2)!}{(m-k)!}((m-k)(m-k-1)x^2+2k(m-k)(x-1)x+k(k-1)(x-1)^2)\\
&=\frac{(m-k)!}{m!} \frac{(m-2)!}{(m-k)!} \left( m(m-1)x^2 -2k(m-1)x + k(k-1)\right).
\end{aligned}
\]
Then a simple calculation shows that 
\[
\lambda_{\rm min}(f_\emptyset)=\frac{k-\sqrt{k(m-k)/(m-1)}}{m}.
\]
Employing a methodology analogous to that used in the proof of Theorem \ref{th:main}, we demonstrate the existence of a subset $S_0 \subset [m]$
  of size $k$, such that:
\[
 \sigma_{\rm min}(\vA_{S_0})^2 \geq \lambda_{\rm min}(f_\emptyset) = \frac{k-\sqrt{k(m-k)/(m-1)}}{m}=g_1(k).
 \]

 For the case where $n=3$, Lemma \ref{le:fempty} implies
 \[
\begin{aligned}
f_{\emptyset}(x)=&-\frac{(m-3)!}{(m-k)!}\left(k(k-1)(k-2)(1-x)^3-3k(k-1)(m-k)x(1-x)^2\right.\\
 & \left. \qquad +3k(m-k)(m-k-1)x^2(1-x)-(m-k)(m-k-1)(m-k-2)x^3\right)\\
=&\frac{(m-3)!}{(m-k)!} \left(m(m-1)(m-2)x^3-3k(m-1)(m-2)x^2+3k(k-1)(m-2)x-k(k-1)(k-2)\right).
\end{aligned}
\]

 Employing the  three real-root formula due to François Viète \cite{3ci}, we obtain 
 \[
 \lambda_{\rm min}(f_\emptyset)
=\frac{k}{m}+\frac{2}{m}\sqrt{\frac{k(m-k)}{m-1}}\cos \left(\frac{1}{3} \arccos\left(\frac{m-2k}{m-2}\sqrt{\frac{m-1}{k(m-k)}}\right) +\frac{2\pi}{3}\right)=g_2(k).
 \]
  As in the previous analysis, we establish the existence of a subset $S_0 \subset [m]$
  with cardinality $k$, satisfying:
\[
 \sigma_{\rm min}(\vA_{S_0})^2 \geq g_2(k).
 \]

 For cases where $k \in \{2,3\}$ and $n \geq k$, it suffices to present the value of $r_k(f_\emptyset)$.
 For $k=2$, as established in Lemma \ref{le:gmin}, we have $\lambda_{\rm min}(g_\emptyset) = r_k(f_\emptyset)$, where
 \[
g_{\emptyset}(x)=n(n-1)(1-x)^2+2n(m-n)x(x-1)+(m-n)(m-n-1)x^2.
\]
 A straightforward calculation yields:
\[
r_2(f_\emptyset)=\lambda_{\rm min}(g_\emptyset)=\frac{n-\sqrt{n(m-n)/(m-1)}}{m}=g_1(n).
\]

For the case where $k=3$, we have 
\[
g_\emptyset(x)=-m(m-1)(m-2)x^3+3(m-1)(m-2)nx^2-3(m-2)n(n-1)x+n(n-1)(n-2).
\]
 Applying the three real-root formula attributed to François Viète \cite{3ci}, we obtain:
\[
\begin{aligned}
r_3(f_\emptyset)=\lambda_{\rm min}(g_\emptyset)
= \frac{n}{m}+\frac{2}{m}\sqrt{\frac{n(m-n)}{m-1}}\cos \left(\frac{1}{3} \arccos\left(\frac{m-2n}{m-2}\sqrt{\frac{m-1}{n(m-n)}}\right) +\frac{2\pi}{3}\right)=g_2(n).
\end{aligned}
\]
The existence of a subset $S_0 \subset [m]$
  with cardinality $k$ that satisfies the conditions stated in the theorem follows as a direct consequence.
\end{proof}

\section*{Acknowledgments}
The author would like to express his sincere gratitude to  Prof. Jiang Yang from Southern University of Science and Technology for valuable discussions, which notably directed his attention to reference \cite{HongPan}. This insightful exchange took place during the ``{\em Frontiers in Scientific Computing and Numerical Analysis Workshop}" organized by Prof. Huazhong Tang and Prof. Kailiang Wu at Southern University of Science and Technology from March 28 to March 30, 2025.

\end{document}